\documentclass[leqno]{amsart}
\usepackage{amsmath,amsfonts,amsthm,amssymb,indentfirst,epic,url}

\newtheorem{theorem}{Theorem}
\newtheorem{lemma}[theorem]{Lemma}
\newtheorem{corollary}[theorem]{Corollary}
\newtheorem{proposition}[theorem]{Proposition}
\newtheorem{definition}[theorem]{Definition}

\newcommand{\Z}{\mathbb Z}
\newcommand{\U}{\mathcal{U}}
\newcommand{\V}{\mathbf{V}}
\newcommand{\F}{\mathcal{F}}
\renewcommand{\r}{\mathrm}
\newcommand{\card}{\r{card}}

\begin{document}

\begin{center}
\texttt{Comments, corrections,
and related references welcomed, as always!}\\[.5em]
{\TeX}ed \today
\vspace{2em}
\end{center}

\title%
[Families of ultrafilters and homomorphisms on products]%
{Families of ultrafilters, and homomorphisms\\
on infinite direct product algebras}
\thanks{
This preprint is readable online at
\url{http://math.berkeley.edu/~gbergman/papers/}.
}


\author{George M. Bergman}
\address{University of California\\
Berkeley, CA 94720-3840, USA}
\email{gbergman@math.berkeley.edu}

\begin{abstract}
Criteria are obtained for a filter $\F$ of subsets of a
set $I$ to be an intersection of finitely many ultrafilters,
respectively, finitely many $\!\kappa\!$-complete ultrafilters
for a given uncountable cardinal $\kappa.$
From these, general results are deduced concerning homomorphisms
on infinite direct product groups, which yield quick proofs of
some results in the literature: the {\L}o\'{s}-Eda
theorem (characterizing homomorphisms from a not-necessarily-countable
direct product of modules to a slender module), and some results of
N.\,Nahlus and the author on homomorphisms on infinite direct
products of not-necessarily-associative $\!k\!$-algebras.
The same tools allow other results of Nahlus and the author to
be nontrivially strengthened, and yield
an analog to one of their results, with nonabelian groups
taking the place of $\!k\!$-algebras.

We briefly examine the question of how the
common technique used in applying the general results of this note
to $\!k\!$-algebras on the one
hand, and to nonabelian groups on the other,
might be extended to more general varieties of algebras in the
sense of universal algebra.

In a final section, the Erd\H{o}s-Kaplansky Theorem on dimensions
of vector spaces $D^I$ $(D$ a division ring)
is extended to reduced products $D^I/\F,$ and an application is noted.
\end{abstract}
\maketitle

\section{Results on filters and ultrafilters}\label{S.cap_ultra}

The definition below recalls some standard concepts.
Readers not familiar with some of these might skim those
parts of the definition
now, and return to them as one or another concept is called on.
(For a thorough development of ultrafilters and
related topics, see works such as \cite{C+N} or~\cite{Drake}.)

\begin{definition}\label{D.filter&&}
If $I$ is a set, then a \emph{filter} on $I$ means a
set $\F$ of subsets of $I,$ such that
\textup{(i)}~$I\in\F,$
\textup{(ii)}~if $J\in\F$ and $J\subseteq K\subseteq I,$ then $K\in\F,$
and \textup{(iii)}~if $J,\,K\in\F,$ then $J\cap K\in\F.$
A filter $\F$ on $I$ is \emph{proper} if it is not the set
of all subsets of $I,$ equivalently, if $\emptyset\notin\F.$
A filter $\F$ on $I$ is \emph{$\!\kappa\!$-complete,}
for $\kappa$ an infinite cardinal, if $\F$ is closed under
intersections of families of $<\kappa$ elements.
\textup{(}Thus, every filter is $\!\aleph_0\!$-complete.\textup{)}
A filter which is $\!\aleph_1\!$-complete, i.e., closed
under countable intersections, is called
\emph{countably complete}.

A maximal proper filter on $I$ is called an \emph{ultrafilter}.
An ultrafilter of the form $\{J\subseteq I\mid i_0\in J\}$ for some
$i_0\in I$ is called \emph{principal}; all other ultrafilters
are called \emph{nonprincipal}.

An infinite cardinal $\kappa$ is called \emph{measurable} if
there exists a nonprincipal $\!\kappa\!$-complete ultrafilter
on $\kappa.$
\end{definition}

The use to which filters will be put in this note
arises from the following observation.

\begin{lemma}\label{L.A->C}
Let $I$ be a set, $(A_i)_{i\in I}$ an $\!I\!$-indexed family
of nonempty sets, and $h:\prod_I A_i\to C$ a map
from their direct product to another set.
Then the set
\begin{equation}\begin{minipage}[c]{37pc}\label{d.A->C}
$\F\,=\,\{J\subseteq I\mid h$ factors
as $\prod_{i\in I} A_i\to \prod_{i\in J} A_i\to C,$
where the first map is the natural projection$\}$
\end{minipage}\end{equation}
is a filter on $I.$

Conversely, given any filter $\F$ on $I,$ and any $\!I\!$-indexed
family $(A_i)_{i\in I}$ of sets each having more than one
element, there exists a map $h$ from $\prod_I A_i$ to a set $C,$
such that the filter $\F$ is given by\textup{~(\ref{d.A->C})}.
\end{lemma}

\begin{proof}
It is clear that the set~(\ref{d.A->C})
satisfies conditions~(i) and~(ii) in the first sentence of
Definition~\ref{D.filter&&}.
To see~(iii), note that if $h$ factors through the subproducts
indexed by $J$ and by $K,$ this means $h(a)$ is unaffected
by any change in $a$ that either modifies only coordinates
outside of $J,$ or modifies only coordinates outside of $K.$
Now a change affecting only coordinates outside
$J\cap K$ can be achieved by first modifying coordinates outside
$J,$ then coordinates in $J-K.$
Hence if neither of these alters $h(a),$ the combined change doesn't,
so $J\cap K\in\F,$ as required.

To get the converse, one takes for $C$ the \emph{reduced product}
$(\prod_I A_i)/\F,$ that is, the factor-set
of $\prod A_i$ by the relation making
$(a_i)_{i\in I}\sim (b_i)_{i\in I}$ if $\{i\mid a_i=b_i\}\in\F.$
It is straightforward to check that this is an equivalence relation,
and (remembering that each $A_i$ has more than one
element), that the filter induced by the factor
map $\prod A_i\to(\prod A_i)/\F$ is precisely~$\F.$
\end{proof}

It is easily shown that a filter $\F$ on $I$ is an ultrafilter
if and only if for every $J\subseteq I,$ either
$J\in\F$ or $I-J\in\F;$ equivalently, if and only if
for all $J,\,K\subseteq I,$ if $J\cup K\in\F$
then either $J\in\F$ or $K\in\F;$ and that the filters on $I$
are precisely the \emph{intersections} of sets of ultrafilters.
(In this last statement, I am following the convention that,
among sets of subsets of
$I,$ we regard the set of all subsets, i.e., the improper filter,
as the intersection of the empty family of sets of subsets.
If we did not allow the empty intersection, then the intersections
of ultrafilters would be the \emph{proper} filters on $I.)$

The next result characterizes those filters that are
intersections of \emph{finitely many} ultrafilters.
In the statement, a \emph{partition} of a set means an expression
of it as the union of a family of pairwise disjoint subsets.
We do not require these subsets to be nonempty, so under our definition,
a partition may involve one or more occurrences of the empty set.

\begin{lemma}\label{L.finiteU}
Let $I$ be a set, and $\F$ a filter on $I.$
Then the following conditions are equivalent.
\begin{equation}\begin{minipage}[c]{35pc}\label{d.1_of_inf}
For every partition of $I$ into infinitely many subsets
$J_s$ $(s\in S,$ $S$ an infinite set\textup{)}
there is at least one $s\in S$ such that $I-J_s\in\F.$
\end{minipage}\end{equation}
\begin{equation}\begin{minipage}[c]{35pc}\label{d.1_of_ctb}
For every partition of $I$ into a \emph{countably} infinite
family of subsets $J_m$ $(m\in\omega),$ there is at least one
$m\in\omega$ such that $I-J_m\in\F.$
\end{minipage}\end{equation}
\begin{equation}\begin{minipage}[c]{35pc}\label{d.bdd}
There exists $n\in\omega$ such that for every partition
of $I$ into $n+1$ subsets $J_0,\dots,J_n,$ there is at least one
$m\in n+1$ such that $I-J_m\in\F.$
\end{minipage}\end{equation}
\begin{equation}\begin{minipage}[c]{35pc}\label{d.cap_fin}
$\F$ is the intersection of finitely many ultrafilters on $I.$
\end{minipage}\end{equation}

When these conditions hold, the finite set of
ultrafilters having $\F$ as intersection is unique,
and its cardinality is the least $n$ as in~\textup{(\ref{d.bdd})}.
\end{lemma}

\begin{proof}
We shall prove (\ref{d.cap_fin})$\implies$(\ref{d.bdd})%
$\implies$(\ref{d.1_of_inf})$\implies$(\ref{d.1_of_ctb})%
$\implies$(\ref{d.cap_fin}), then the final sentence.

To see (\ref{d.cap_fin})$\implies$(\ref{d.bdd}), let $\F$ be an
intersection of $n$ ultrafilters, and note that two disjoint
sets cannot belong to a common ultrafilter.
Hence in any
partition of $I$ into $n+1$ sets, at least one will belong to none
of our $n$ ultrafilters; hence its complement belongs to all of them,
hence to $\F.$
For (\ref{d.bdd})$\implies$(\ref{d.1_of_inf}), take
$n$ as in~(\ref{d.bdd}), partition the infinite
index-set $S$ of~(\ref{d.1_of_inf}) into $n+1$ nonempty
subsets $S_0,\dots,S_n,$ and for $m=0,\dots,n,$ let
$J_m=\bigcup_{s\in S_m} J_s.$
By~(\ref{d.bdd}), the complement of one of the $J_m$ lies in $\F.$
Hence, taking any $s\in S_m,$ the complement of $J_s,$
an overset of the complement of $J_m,$ also lies in $\F.$
(\ref{d.1_of_inf})$\implies$(\ref{d.1_of_ctb}) is clear.

We shall prove (\ref{d.1_of_ctb})$\implies$(\ref{d.cap_fin})
in contrapositive form,
$\neg\!$(\ref{d.cap_fin})$\implies~\neg\!$(\ref{d.1_of_ctb}):

Since $\F$ is a filter, it is the intersection of a
set $U$ of ultrafilters.
Suppose $U$ were infinite.
Take any two distinct members of $U.$
Then there is a subset of $I$ belonging to one but not to the
other; hence the complement of that subset belongs to
the other ultrafilter.
Every ultrafilter on $I$ must contain one of these two sets,
so at least one of them belongs to infinitely many members of $U.$
Let us write $I_0$ for such a one
(making an arbitrary choice if both do), and let $J_0=I-I_0,$
recalling that this still belongs to at least one member of $U.$

We now repeat the process on $I_0,$ decomposing it into
a subset $I_1$ which belongs to infinitely many members
of $U$ and a complementary subset $J_1$ which belongs
to at least one; then repeat the process on $I_1,$ and so forth.
We thus get a countably infinite family
$J_0,~J_1,\dots$ of disjoint subsets
of $I,$ each of which belongs to a member of $U.$
If $\bigcup J_m\neq I,$ we enlarge one of them,
say $J_0,$ by attaching $I-\bigcup J_m$ to it.
We then have a partition of $I$ into sets
$J_i$ each belonging to a member of~$U.$
Hence none of their complements
belongs to all members of $U,$ i.e., belongs to $\F,$
proving~$\neg\!$(\ref{d.1_of_ctb}).

To get the final sentence, use~(\ref{d.cap_fin})
to write $\F=\U_0\cap\dots\cap\U_{n-1}$ with the $\U_m$ distinct.
For any ultrafilter $\U$ distinct from each of the $\U_m,$
we can find sets $J_m\in\U_m-\U$ $(m=0,\dots,n-1).$
Hence $J_0\cup\dots\cup J_{n-1}$ belongs
to all $\U_m$ but not to $\U,$ showing that $\F\not\subseteq\U.$
Thus any other set of ultrafilters with intersection $\F$
must be a subset of $\{\U_0,\dots,\U_{n-1}\};$ and
reversing the roles of the two sets of ultrafilters, we get equality.
Our proof of~(\ref{d.cap_fin})$\implies$(\ref{d.bdd})
showed that this $n$ can be used as the $n$ of~(\ref{d.bdd}).
On the other hand, the conclusion
of~(\ref{d.bdd}) does not hold for any smaller value than~$n,$
since we can partition $I$ into $n$ sets, one in each $\U_m.$
\end{proof}

(One can get still more conditions equivalent to those of
the above lemma by replacing
the partition of $I$ in each of~(\ref{d.1_of_inf})-(\ref{d.bdd})
either by a family of disjoint subsets $J_s$ of $I,$ or by a family of
sets $J_s$ having union~$I.$
In the former case, one keeps the conclusions as
in~(\ref{d.1_of_inf})-(\ref{d.bdd});
in the latter, one replaces them by statements that the
union of all but one of the sets $J_s$ lies in $\F.$
These conditions are easily shown equivalent
to~(\ref{d.1_of_inf})-(\ref{d.bdd}),
using the observation that the members of any
family of disjoint subsets of $I$ can be enlarged so that
they give a partition, and the members of any family
with union $I$ can be shrunk down to give a partition.
Two more equivalent conditions, of a different flavor, are proved
toward the end of this note, in Lemma~\ref{L.red_finite}.)

The next lemma gives a condition for
the finitely many ultrafilters of Lemma~\ref{L.finiteU}
to be $\!\kappa\!$-complete, for a
specified uncountable cardinal $\kappa.$
We recall that a $\!\kappa\!$-complete ultrafilter on $I$ can be
nonprincipal only if $I$ has cardinality at least some
measurable cardinal $\geq\kappa$ \cite[Proposition~4.2.7]{Ch+Keis}.
(Following \cite{Ch+Keis}, I have worded Definition~\ref{D.filter&&}
so that $\aleph_0$ counts as a measurable cardinal.
I therefore write ``uncountable measurable cardinal''
for what many authors simply call a measurable cardinal.)
It is known that if uncountable measurable cardinals exist,
they are very large, and very rare; in particular, that
if the standard set-theory, ZFC, is consistent, then it
is consistent with the nonexistence of such cardinals
\cite[Chapter~6, Corollary~1.8]{Drake}.
Thus, under weak assumptions on the size of $I,$ or
reasonable assumptions on our set theory, the
ultrafilters of the next lemma must be principal.
In the proof of that lemma, we will use the fact that
an ultrafilter $\U$ on $I$ is $\!\kappa\!$-complete if and only
if for every partition of $I$ into $<\kappa$ subsets, one
of these subsets lies in $\U.$

\begin{lemma}\label{L.kappa}
Let $I$ be a set, $\F$ a filter on $I,$
and $\kappa$ an uncountable cardinal.
Then the following statements are equivalent.
\begin{equation}\begin{minipage}[c]{35pc}\label{d.most_of_<kappa}
For every partition of $I$ into $<\kappa$ subsets $J_s$
$(s\in S),$ there exist finitely many indices
$s_0,\dots,s_{n-1}\in S$ such that
$J_{s_0}\cup\dots\cup J_{s_{n-1}}\in\F.$
\end{minipage}\end{equation}
\begin{equation}\begin{minipage}[c]{35pc}\label{d.cap_fin_k}
$\F$ is the intersection of finitely many $\!\kappa\!$-complete
ultrafilters.
\end{minipage}\end{equation}
\end{lemma}

\begin{proof}
Assuming~(\ref{d.cap_fin_k}), let
$\F=\U_0\cap\dots\cap\U_{n-1}$ with all $\U_m$ $\!\kappa\!$-complete.
Given a partition of $I$ into sets $J_s$ as in~(\ref{d.most_of_<kappa}),
$\!\kappa\!$-completeness implies that
each $\U_m$ contains one $J_s;$ say $J_{s_m}\in\U_m.$
Thus $J_{s_0}\cup\dots\cup J_{s_{n-1}}\in\F.$

Conversely, assume~(\ref{d.most_of_<kappa}).
Since $\kappa$ is uncountable,~(\ref{d.most_of_<kappa})
applies in particular to countable decompositions, hence
implies~(\ref{d.1_of_ctb}), which
is equivalent to~(\ref{d.cap_fin}),
i.e., to~(\ref{d.cap_fin_k}) without the
specification of $\!\kappa\!$-completeness.
Now suppose some ultrafilter $\U\supseteq \F$ were
not $\!\kappa\!$-complete.
Then there would exist a partition of $I$ into fewer than $\kappa$
subsets $J_s\notin\U.$
The union of any finite subfamily of these is still $\notin\U,$
hence $\notin\F,$ so~(\ref{d.most_of_<kappa}) fails.
This contradiction completes the proof.
\end{proof}

(Incidentally, the condition on a filter $\F$ that one might
naively hope would imply that $\F$ is an intersection of
$\!\kappa\!$-complete ultrafilters -- namely, that $\F$ itself be
$\!\kappa\!$-complete -- definitely does not.
E.g., if $\kappa$ is a regular infinite non-measurable
(hence uncountable) cardinal,
then the filter $\F$ of complements in $\kappa$ of subsets of
cardinality $<\kappa$ is $\!\kappa\!$-complete, but
there are no nonprincipal $\!\kappa\!$-complete ultrafilters
on $\kappa.$
A cardinal $\kappa$ such that every $\!\kappa\!$-complete
filter extends to a $\!\kappa\!$-complete ultrafilter
is called \emph{strongly compact}; cf.\ \cite{JB+MM}.)

Digression: If $U$ is a (not necessarily finite) set of ultrafilters
on a set $I,$ then the four sets
\begin{equation}\begin{minipage}[c]{35pc}\label{d.4}
$\F\,=\,\bigcap_{\,\U\in U}\U,\qquad
\mathcal{G}\,=\,\bigcup_{\,\U\in U}\U,\qquad
\mathcal{H}\,=\,\bigcup_{\,\U\in U}{^\r{c}\kern.08em\U},\qquad
\mathcal{I}\,=\,\bigcap_{\,\U\in U}{^\r{c}\kern.08em\U},$
\end{minipage}\end{equation}
though they do not, in general, uniquely determine $U,$ do
all determine one another.
Indeed, on the one hand, $\F$ and $\mathcal{H}$ are complements
of one another, as are $\mathcal{G}$ and $\mathcal{I}.$
On the other hand, from the fact that for any ultrafilter $\U,$
the complement of $\U$ is
the set of complements of members of $\U$ (in $I),$
one sees that $\mathcal{I}$ is the set of complements of
members of $\F,$ and vice versa.
(This makes each of $\F$ and $\mathcal{G}$ the sets of complements of
members of the other's complement.
It is not hard to show that
each can also be described as the set of subsets of $I$ having
nonempty intersection with all members of the other.
Likewise, $\mathcal{H}$ and $\mathcal{I},$
in addition to being the sets of complements of members of
each other's complements, are each the set of
subsets of $I$ whose union with every member of the other
is a proper subset of $I.)$

The description of filters in Definition~\ref{D.filter&&}
translates into equally elementary characterizations of the sorts of
sets that can occur as $\mathcal{G},$ $\mathcal{H}$ and $\mathcal{I}$
in~(\ref{d.4}); and since each set in~(\ref{d.4})
conveys the same information, each of these sorts of
sets can, mutatis mutandis,
serve the same mathematical function as filters.
Sets of the form $\mathcal{I}$ are called \emph{ideals} of
subsets of $I,$ since they are the ideals in the Boolean ring
of all its subsets.
Sets having the form $\mathcal{G}$
were named \emph{grills} in \cite{choquet} (cf.\ \cite{thron}), and
are sometimes used under that name in topological contexts.

When I first obtained the results of this note, I formulated
them in terms of finite unions $\mathcal{G}$ of ultrafilters.
I finally realized that what I was doing could be restated
in terms of filters, and rewrote the note accordingly, since
filters are the most familiar of these four sorts of structures.

\section{Ultrafilters, and maps on direct products}\label{S.A->C}

Suppose, as in Lemma~\ref{L.A->C}, that $h: A = \prod_I A_i\to C$
is a map on a direct product of nonempty sets, and
$\F$ the filter of subsets of $I$ corresponding to
those sub-products through which the map factors.
Thus, $h$ factors in a natural way through the canonical
map $A\to A/\F,$
where $A/\F$ denotes the reduced product of the $A_i$ with respect
to $\F,$ defined in the last paragraph of the proof of that lemma.
(The factoring map $A/\F\to C$ is not in general one-to-one.
E.g., if $I=\{0,1\},$ $A_0=A_1=C$ is a nontrivial abelian group $G,$
and $h:G\times G\to G$ its group operation,
then $\F$ is the trivial filter $\{I\},$
so $A\to A/\F$ is an isomorphism, hence $A/\F\to C$ is not one-to-one.)

Now suppose we write the filter $\F$ as $\bigcap_{\,\U\in U} \U$
for $U$ some set of ultrafilters on $I.$
Can we similarly factor $h$ through the natural map
$A\to\prod_{\U\in U}A/\U$?

Yes; but this time not, in general, in a natural way.
Elements of $A$ fall together in $\prod_{\U\in U}A/\U$ if and only if
they fall together in $A/\F,$ but $\prod_{\U\in U}A/\U$ is
typically much larger than the embedded image of $A/\F.$
One can extend the induced map from the image of $A/\F$ to $C$
to a map on all of $\prod A/\U$ by letting
it act in arbitrary ways on elements not in
that image; but there is no guarantee that such an extension
can be made to respect further structure on our sets,
e.g., structures of group or of algebra.

Let us now show, however, that in the context of
Lemmas~\ref{L.finiteU} and~\ref{L.kappa}, where we have
only finitely many ultrafilters, the image of $A,$ and
hence of $A/\F,$ is the full product $\prod_{\U\in U} A/\U,$ so
that the above problem does not arise.

\begin{lemma}\label{L.surj}
Let $I$ be a set, $(A_i)_{i\in I}$ an $\!I\!$-tuple of nonempty sets,
and $\U_0,\dots,\U_{n-1}$ $(n\in\omega)$ distinct ultrafilters on $I.$
Then the natural map
$A=\prod A_i\to A/\U_0\times\dots\times A/\U_{n-1}$ is surjective;
equivalently, the natural
embedding $A/(\U_0\cap\dots\cap\U_{n-1})\hookrightarrow
A/\U_0\times\dots\times A/\U_{n-1}$ is a bijection.
\end{lemma}

\begin{proof}
Since the $\U_m$ are distinct, we can find a
partition $I=J_0\cup\dots\cup J_{n-1}$ with each $J_m\in\U_m.$
Now given any
$(x_0,\dots,x_{n-1})\in A/\U_0\times\dots\times A/\U_{n-1},$
let us choose a representative $a^{(m)}\in A$ of each $x_m,$
and let $a\in A$ be the element which agrees on
each $J_m$ with $a^{(m)}.$
This will map to $(x_0,\dots,x_{n-1})$ in
$A/\U_0\times\dots\times A/\U_{n-1},$ as desired, proving the first
assertion.
The equivalence of this with the second assertion is clear.
\end{proof}

In most of the remainder of this note,
the $A_i$ and $C$ of Lemma~\ref{L.A->C} will have, inter alia, group
structures (usually abelian and written additively).
In this situation, let us define the {\em support} of
an element $a=(a_i)_{i\in I}\in A$ as the set
\begin{equation}\begin{minipage}[c]{35pc}\label{d.supp}
$\r{supp}(a)\ =\ \{i\in I\mid a_i\neq 0\}$\quad (or if our groups
are written multiplicatively, $\{i\in I\mid a_i\neq e\}).$
\end{minipage}\end{equation}

Then we can make the natural identifications,
\begin{equation}\begin{minipage}[c]{35pc}\label{d.prod_J}
For $J\subseteq I,$ we identify the subalgebra
$\{a\in\prod_{i\in I} A_i\mid\r{supp}(a)\subseteq J\}\subseteq A$
with $\prod_{i\in J} A_i.$
\end{minipage}\end{equation}

Note that for $a,\,a'\in A,$ the set of indices at which
these two elements differ can be described as the support
of $a-a'$ (respectively $aa'^{-1}).$
Combining this observation with the identification~(\ref{d.prod_J}),
we see that in the
context of Lemma~\ref{L.A->C}, if the $A_i$
and $C$ are groups and $h$ a homomorphism, then~(\ref{d.A->C}) becomes
\begin{equation}\begin{minipage}[c]{35pc}\label{d.{supp}}
$\F\ =
\ \{J\subseteq I\ \mid\ \prod_{i\in I-J} A_i\,\subseteq\,\r{ker}(h)\}.$
\end{minipage}\end{equation}

Let us now apply Lemma~\ref{L.finiteU} to the above situation.
We could give a translation of each of
conditions~(\ref{d.1_of_inf})-(\ref{d.cap_fin}),
but for brevity, we focus on~(\ref{d.1_of_ctb}) and~(\ref{d.cap_fin}).

\begin{corollary}[to Lemma~\ref{L.finiteU}]\label{C.finiteU}
Suppose $I$ is a set, $(A_i)_{i\in I}$ a family of groups,
$C$ a group, and $h:A=\prod_I A_i\to C$ a group homomorphism.
Then the following conditions are equivalent:
\begin{equation}\begin{minipage}[c]{35pc}\label{d.ann_1_of_ctb}
For every partition of $I$ into a countably infinite
family of subsets $J_m$ $(m\in\omega),$ at least one of the
subgroups $\prod_{i\in J_m} A_i\subseteq A$ lies in $\r{ker}(h).$
\end{minipage}\end{equation}
\begin{equation}\begin{minipage}[c]{35pc}\label{d.f_factors}
The homomorphism $h$ factors
$A\to A/\U_0\times\dots\times A/\U_{n-1}\to C$
\textup{(}where the first arrow is the product of the
quotient maps\textup{)}
for some finite family of ultrafilters $\U_0,\dots,\U_{n-1}$ on $I.$
\end{minipage}\end{equation}

In this situation, the filter $\F$ of~\textup{(\ref{d.{supp}})} is the
intersection of the unique least set of ultrafilters that can be
used in~\textup{(\ref{d.f_factors})}.
\end{corollary}

\begin{proof}
Defining $\F$ by~(\ref{d.{supp}}), equivalently,
by~(\ref{d.A->C}), Lemma~\ref{L.A->C} tells us that $\F$
is a filter on $I.$
Condition~(\ref{d.ann_1_of_ctb}) then translates to~(\ref{d.1_of_ctb}),
which by Lemma~\ref{L.finiteU} is equivalent to~(\ref{d.cap_fin}),
i.e., the condition that $\F$ is an
intersection of finitely many ultrafilters, $\U_0\cap\dots\cap\U_{n-1}.$
Let us show that such an expression for $\F$ is equivalent to a
factorization of $h$ as in~(\ref{d.f_factors}).

On the one hand, if $\F=\U_0\cap\dots\cap\U_{n-1},$
then by Lemma~\ref{L.surj} and the discussion preceding it,
$h$ has the desired factorization (as a group homomorphism).
Conversely, given~(\ref{d.f_factors}),
an element of $A$ whose support lies in none of the $\U_m$ will
belong to $\r{ker}(h),$ hence by~(\ref{d.{supp}}),
$\U_0\cap\dots\cap\U_{n-1}\subseteq\F.$
Hence, as in the last paragraph of
the proof of Lemma~\ref{L.finiteU},
$\U_0,\dots,\U_{n-1}$ are the only ultrafilters that contain
$\F;$ so since we know it is an intersection of ultrafilters,
it must be the intersection of some subset of this finite family.
Hence it is, as required, a finite intersection of ultrafilters.
By the final sentence of Lemma~\ref{L.finiteU}, the resulting
set of ultrafilters is unique, and we get
the final sentence of the present lemma.
\end{proof}

Combining the above with Lemma~\ref{L.kappa}, we likewise get

\begin{corollary}[to Lemma~\ref{L.kappa}]\label{C.kappa}
Suppose $I$ is a set, $(A_i)_{i\in I}$ a family of groups,
$C$ a group, $h:A=\prod_I A_i\to C$ a group homomorphism,
and $\kappa$ an uncountable cardinal.
Then the following conditions are equivalent:
\begin{equation}\begin{minipage}[c]{35pc}\label{d.ann_most_of_<kappa}
For every partition of $I$ into $<\kappa$ subsets $J_s$
$(s\in S),$ there exist finitely many indices
$s_0,\dots,s_{n-1}\in S$ such that
$\prod_{i\in I-J_{s_0}\cup\dots\cup J_{s_{n-1}}}A_i\subseteq
\r{ker}(h).$
\end{minipage}\end{equation}
\begin{equation}\begin{minipage}[c]{35pc}\label{d.f_factors+kappa}
The homomorphism $h$ factors
$A\to A/\U_0\times\dots\times A/\U_{n-1}\to C$ for some finite family
of $\!\kappa\!$-complete ultrafilters $\U_0,\dots,\U_{n-1}$ on $I.$
Hence, if $\card(I)$ is less than every
$\!\kappa\!$-complete measurable cardinal \textup{(}in particular,
if no such cardinals exist\textup{)}, then $h$ factors through the
projection of $A$ to the product of finitely many of the $A_i.$
\end{minipage}\end{equation}

Again, the $\F$ of~\textup{(\ref{d.{supp}})} is the intersection
of the least family of $\!\kappa\!$-complete ultrafilters
as in~\textup{(\ref{d.f_factors+kappa})}.\qed
\end{corollary}

\section{Strengthening results from \cite{prod_Lie2}}\label{S.prod_alg}

If $k$ is a commutative ring (by which we will always
mean a commutative associative unital ring),
then a \emph{$\!k\!$-algebra} (often shortened to ``an algebra''
when there is no danger of ambiguity) will here mean a
$\!k\!$-module $A$ given with a $\!k\!$-bilinear
map $A\times A\to A,$ written as multiplication, but
{\em not} assumed associative, or commutative, or unital.
As in \cite{prod_Lie1} and \cite{prod_Lie2},
for $A$ an algebra, we define its {\em total annihilator ideal} by
\begin{equation}\begin{minipage}[c]{35pc}\label{d.Z(A)}
$Z(A)~\,=~\,\{x\in A\mid xA=Ax=\{0\}\}.$
\end{minipage}\end{equation}

Given algebras $A_i$ $(i\in I)$ and $B,$ and
a surjective homomorphism $f:A=\prod_I A_i\to B,$
N.\,Nahlus and the present author study in \cite{prod_Lie1}
and \cite{prod_Lie2} conditions under which
\begin{equation}\begin{minipage}[c]{35pc}\label{d.f=}
$f$ can be written as the sum,
$f_1+f_0,$ of a $\!k\!$-algebra homomorphism $f_1:A\to B$ that
factors through the projection of $A$ onto the product
of finitely many of the $A_i,$ and a $\!k\!$-algebra
homomorphism $f_0:A\to Z(B).$
\end{minipage}\end{equation}
In particular,
\begin{equation}\begin{minipage}[c]{35pc}\label{d.prod_Lie2}
(From \cite[Theorem~9]{prod_Lie2}) \ Given
a surjective homomorphism $f:A=\prod_I A_i\to B$
of algebras over an infinite field $k,$ a decomposition $f=f_0+f_1$
as in~(\ref{d.f=}) exists if either\\[.1em]
(i)\ \ $\dim_k(B)<\card(k),$ and
$\card(I)\leq\card(k),$ or \\[.1em]
(ii)\,\ $\dim_k(B)<2^{\aleph_0},$ and
$\card(I)=\aleph_0,$ or\\[.1em]
(iii)\ $\dim_k(B)$ is finite, and $\card(I)$ is less than
every measurable cardinal $>\card(k).$
\end{minipage}\end{equation}

It occurred to me (while correcting the galley proofs to
\cite{prod_Lie2}!)
that even if $I$ does \emph{not} satisfy one of the cardinality bounds
of~(i)--(iii) above, one can look at partitions
\begin{equation}\begin{minipage}[c]{35pc}\label{d.partition}
$I\ =\ \bigcup_{s\in S} J_s$
\end{minipage}\end{equation}
where $S$ does satisfy that bound,
and apply~(\ref{d.prod_Lie2}) to the resulting product expressions
\begin{equation}\begin{minipage}[c]{35pc}\label{d.prod_prod}
$A\ =\ \prod_{s\in S}\,(\prod_{i\in J_s} A_i).$
\end{minipage}\end{equation}

That approach yielded improvements on~(\ref{d.prod_Lie2}),
culminating in the present note.
The set-theoretic arguments underlying those improvements have been
abstracted in \S\S\ref{S.cap_ultra}-\ref{S.A->C} above.
Combining those with~(\ref{d.prod_Lie2}), we can now get

\begin{theorem}[strengthening of {(\ref{d.prod_Lie2})}]\label{T.>prod_Lie2}
Suppose $k$ is an infinite field, $(A_i)_{i\in I}$ a family
of $\!k\!$-algebras, $B$ a $\!k\!$-algebra, and
$f:A=\prod_I A_i\to B$ a surjective $\!k\!$-algebra homomorphism.
Suppose also either that\\[.2em]
\textup{(i)}\ \ $\dim_k(B)<\card(k),$ or that\\[.2em]
\textup{(ii)}\,\ $\dim_k(B)<2^{\aleph_0}.$

Then the composite homomorphism
\begin{equation}\begin{minipage}[c]{35pc}\label{d.A->->B/Z}
$A\ \to\ B\ \to\ B/Z(B)$
\end{minipage}\end{equation}
can be factored
\begin{equation}\begin{minipage}[c]{35pc}\label{d.factor}
$A\ \to\ A/\U_0\times\dots\times A/\U_{n-1}\ \to\ B/Z(B)
\quad(n\in\omega),$
\end{minipage}\end{equation}
where the $\U_m$ are ultrafilters on $I,$ which in case~\textup{(i)}
are $\!\card(k)^+\!$-complete, and in case~\textup{(ii)}
countably complete.

Thus, if $\card(I)$ is less than every measurable
cardinal $>\card(k)$ in case~\textup{(i)},
or less than every uncountable measurable cardinal in case~\textup{(ii)}
\textup{(}in particular, in either
case, if no such measurable cardinals exist\textup{)},
then~\textup{(\ref{d.A->->B/Z})}
factors through the projection of $A$ to a finite subproduct
$A_{i_0}\times\dots\times A_{i_{n-1}}$ $(i_0,\dots,i_{n-1}\in I).$
In that situation, the original homomorphism $f:A\to B$ can be
written $f=f_1+f_0$ as in\textup{~(\ref{d.f=})}.
\end{theorem}

\begin{proof}
In case~(i), let $\kappa=\card(k)^+,$ and in case~(ii),
$\kappa=\aleph_1.$
Then in each case,
given any partition of $I$ into $<\kappa$ subsets $J_s,$
if we write $A$ as the product of products~(\ref{d.prod_prod}),
then~(\ref{d.prod_Lie2}) tells us that $f$ decomposes as
the sum of a map that factors through a subproduct
$\prod_{i\in J_{s_0}\cup\dots\cup J_{s_{n-1}}} A_i,$
and another with image in $Z(B).$
Hence, on composing with the factor
map $B\to B/Z(B),$ we get a factorization of~(\ref{d.A->->B/Z})
through a subproduct
$\prod_{i\in J_{s_0}\cup\dots\cup J_{s_{n-1}}} A_i.$
Corollary~\ref{C.kappa} now gives us everything but
the last sentence of the theorem.

To get that sentence let us
(as in \cite[end of proof of Theorem~9]{prod_Lie2})
define $f_1,$ respectively, $f_0,$
to be the maps $A\to B$ obtained by first projecting $A$ to
$\prod_{i=i_0,\dots,i_{n-1}} A_i,$ respectively,
$\prod_{i\in I-\{i_0,\dots,i_{n-1}\}} A_i,$ regarded as
subalgebras of $A,$ and then
(in each case) composing that projection with $f$ on the left.
We see that these composite maps have the required properties.
\end{proof}

\emph{Remark}: Case~(iii) of~(\ref{d.prod_Lie2}) has disappeared
from the above statement.
That case was obtained in \cite{prod_Lie2} by a different
method from~(i) and~(ii), for which a trick that allows one
to get $\!\kappa\!$-complete ultrafilters was easier to see, resulting
in the measurable-cardinal bound in the statement of~(iii).
However, once~(i) is strengthened as above, the resulting
statement majorizes~(iii).

(But I still find striking the property of
vector spaces underlying case~(iii)
of~(\ref{d.prod_Lie2}), namely \cite[Lemma~7]{prod_Lie2},
which implies that for any linear map $f$ from $k^I$
$(I$ infinite) to a finite-dimensional $\!k\!$-vector-space $V,$ there
exist finitely many
$\!\card(k)^+\!$-complete ultrafilters $\U_0,\dots,\U_{n-1}$
such that for every $I'\subseteq I$ belonging to none
of the $\U_m,$ there is a member of $\r{ker}(f)$ with
support containing $I'.$
In contrast, I do not see any way to strengthen
the results about supports of elements in
kernels of maps $k^I\to V$ for larger-dimensional $V,$
\cite[Lemmas~3 and~5]{prod_Lie2},
which underlie cases~(i) and~(ii) of~(\ref{d.prod_Lie2}),
so as to raise the upper bounds on $\card(I)$ to a measurable cardinal.
To get Theorem~\ref{T.>prod_Lie2}, we had to use results
about \emph{algebra} homomorphisms and their composites
with $B\to B/Z(B),$ proved from those lemmas.)

Case~(i) of Theorem~\ref{T.>prod_Lie2} above also subsumes
\cite[Theorem~19]{prod_Lie1}, a result
which has the same measurable-cardinal bound on $\card(I)$
as in Theorem~\ref{T.>prod_Lie2},
but stronger assumptions on $B$ (countable
dimensionality, plus a chain condition).
\vspace{.3em}

Returning to~(\ref{d.prod_Lie2}), the methods by which that
result was obtained in \cite[\S\S1-4]{prod_Lie2} are extended in
of \cite[\S6]{prod_Lie2} to get an almost parallel result for a direct
product of algebras over a valuation ring, \cite[Theorem~14]{prod_Lie2}.
Since a field may be regarded as a valuation ring with trivial
value group, the latter result formally subsumed the former
(except that it did not contain a case~(iii)).
However, since the proof was more difficult -- but could be
shortened by referring to aspects of the
earlier proof -- and the statement was
somewhat more complicated, and algebras over fields
are more familiar than algebras over valuation rings,
the two results were stated separately.
Here, likewise, let us state separately our strengthening
of that result.

\begin{theorem}[strengthening of {\cite[Theorem~14]{prod_Lie2}}]\label{T.>prod_Lie2_val}
Let $R$ be a commutative valuation ring with infinite residue
field $k,$ and $f:A=\prod_I A_i\to B$ a surjective
homomorphism from a direct product of $\!R\!$-algebras to an
$\!R\!$-algebra $B$ which is torsion-free as an $\!R\!$-module.
Let us write $\r{rk}_R(B)$ for the rank of $B$
as an $\!R\!$-module; i.e., the common cardinality of all
maximal $\!R\!$-linearly independent subsets of $B.$
Suppose that either\\[.2em]
\textup{(i)}\ \ $\r{rk}_R(B)<\card(k),$ or\\[.2em]
\textup{(ii)}\,\ $\r{rk}_R(B)<2^{\aleph_0}.$

Then the composite homomorphism
\begin{equation}\begin{minipage}[c]{35pc}\label{d.A->->B/Z_val}
$A\ \to\ B\ \to\ B/Z(B)$
\end{minipage}\end{equation}
can be factored
\begin{equation}\begin{minipage}[c]{35pc}\label{d.factor_val}
$A\ \to\ A/\U_0\times\dots\times A/\U_{n-1}\ \to\ B/Z(B),$
\end{minipage}\end{equation}
where the $\U_m$ are ultrafilters on $I,$ which are
$\!\card(k)^+\!$-complete in case~\textup{(i),} and
countably complete in case~\textup{(ii)}.

Thus, again, if $\card(I)$ is less than every measurable
cardinal $>\card(k)$ in case~\textup{(i)},
or less than every uncountable measurable
cardinal in case~\textup{(ii)}, then~\textup{(\ref{d.A->->B/Z_val})}
factors though the projection of $A$ to the product
of finitely many of the $A_i,$ and $f:A\to B$ can
be written as the sum of a homomorphism $f_1: A\to B$ that factors
through these projections, and a homomorphism $f_0: A\to Z(B).$
\end{theorem}

\begin{proof}
Proved from \cite[Theorem~14]{prod_Lie2} exactly as
Theorem~\ref{T.>prod_Lie2}
is proved from \cite[Theorem~9]{prod_Lie2}.
\end{proof}

\section{Three further applications}\label{S.more_aps}

Subsections~\ref{SS.Los-Eda}-\ref{SS.prod_Lie2} below show how
two results in the literature, which were proved there by arguments
about particular sorts of structures, can be obtained
via the general results of \S\S\ref{S.cap_ultra}-\ref{S.A->C} above.
Subsection~\ref{SS.nonab} notes an analog of the result
of subsection~\ref{SS.prod_Lie2} with $\!k\!$-algebras
replaced by nonabelian groups.

\subsection{The {\L}o\'{s}-Eda Theorem}\label{SS.Los-Eda}
Let $R$ be any associative, unital, not necessarily commutative ring.
By an $\!R\!$-module $M$ we will mean either a
left or a right $\!R\!$-module.
(We will not be writing down the action, so we
do not have to choose between right and
left; but all modules are understood to be on the same side.)

Recall that an $\!R\!$-module $M$ is called \emph{slender}
if every homomorphism from a countable direct
product of $\!R\!$-modules $\prod_{i\in\omega} N_i$
into $M$ annihilates all but finitely many of the $N_i.$
(See \cite{E+L}, \cite[\S\S94-95]{Fuchs}, \cite[Chapter~III]{E+M}.
The classical example is the $\!\Z\!$-module $\Z.)$
In this situation, it is easy to show that such a homomorphism
must in fact factor through the projection to a finite
sub-product, $N_{i_0}\times\dots\times N_{i_{n-1}}$
\cite[Theorem~III.1.2]{E+L}.
Homomorphisms from a direct product of modules indexed by a
{\em not necessarily countable} set to a slender module
have a description analogous to the results of the preceding section:
\begin{equation}\begin{minipage}[c]{35pc}\label{d.Los-Eda}
({\L}o\'{s}-Eda Theorem, \cite{KE}, \cite[Theorem~III.3.2]{E+M})
Any homomorphism from a direct product $N=\prod_{i\in I} N_i$
of $\!R\!$-modules to a slender $\!R\!$-module $M$ factors through
the natural map of $N$ to a finite direct product of
ultraproducts, $N/\U_0\times \dots\times N/\U_{n-1},$ where
$\U_0,\dots,\U_{n-1}$ are countably complete ultrafilters on $I.$
\end{minipage}\end{equation}

This now follows from Corollary~\ref{C.kappa} and
the definition of slender module, via the same ``product of
products'' trick used to deduce the theorems of the preceding section
from results of~\cite{prod_Lie2}.

(For further set-theoretic results about homomorphisms
on direct products of abelian groups, see \cite{JB+MM}.)

\subsection{A result from {\cite{prod_Lie1}}}\label{SS.prod_Lie2}
We shall look next at a result proved by N.\,Nahlus and the
present author in \cite{prod_Lie1}.
The hypothesis was weaker than for the
results of \cite{prod_Lie2} -- no assumption of an infinite
base-field $k,$ and a weaker condition
on the codomain algebra than a bound on its $\!k\!$-dimension -- so we
also get a weaker conclusion, a case of Corollary~\ref{C.finiteU}
rather than Corollary~\ref{C.kappa}.

We need some definitions to formulate the hypothesis.
Given an algebra (defined as in the preceding section)
$B$ over a commutative ring
$R,$ we will say that a pair of ideals $B_0,\,B_1\subseteq B$
are \emph{almost direct factors} of $B$ if they sum to $B,$ and each
is the $\!2\!$-sided annihilator of the other.
We call each such ideal \emph{an almost direct factor of $B.$}
The following three observations are easy.
(The first is a special case of the fact that
the $\!2\!$-sided annihilator of every subset
of $B$ contains $Z(B);$ the second is seen by noting that
if $B_0+B_1=B,$ and $B_0$ annihilates both itself and $B_1,$
then it annihilates $B;$ the third by writing an element
of the annihilator of $B_0+Z(B)$ (resp.\ $B_1+Z(B))$ as $x_0+x_1$
with $x_i\in B_i,$ and noting that $x_0$ (resp.\ $x_1)$
must lie in $Z(B).)$
\begin{equation}\begin{minipage}[c]{35pc}\label{d.ZB_in_almost}
Every almost direct factor of $B$ contains $Z(B).$
\end{minipage}\end{equation}
\begin{equation}\begin{minipage}[c]{35pc}\label{d.not_self}
If an almost direct factor of $B$ is strictly larger than $Z(B),$
it does not annihilate itself.
\end{minipage}\end{equation}
\begin{equation}\begin{minipage}[c]{35pc}\label{d.mut_ann}
Whenever $B$ is the sum of
two mutually annihilating ideals $B_0$ and $B_1,$ the
ideals $B_0+Z(B)$ and $B_1+Z(B)$ are almost direct factors.
\end{minipage}\end{equation}

We shall say that $B$ has \emph{chain condition on almost direct
factors} if every ascending chain of almost direct factors
of $B$ terminates; equivalently,
if every descending chain of such ideals terminates.
(The equivalence follows from the order-reversing
relation between pairs of almost direct factors.)
A trivial but important class of
algebras with chain condition on almost
direct factors are the finite-dimensional algebras over fields.

The result from \cite{prod_Lie1} that we will recover is
\begin{equation}\begin{minipage}[c]{35pc}\label{d.CC}
\cite[part of Proposition~16]{prod_Lie1}\ \ %
If $f: A=\prod_I A_i\to B$ is a surjective homomorphism of
algebras over a commutative ring
$R,$ and $B$ has chain condition on almost direct factors,
then there exist finitely
many ultrafilters $\U_0,\dots,\U_{n-1}$ on $I$ such that
the composite map $A\to B\to B/Z(B)$ factors through the
natural map $A\to A\,/\,\U_0\times\dots\times A\,/\,\U_{n-1}.$
\end{minipage}\end{equation}

To get this, we shall show that the
composite map $A\to B\to B/Z(B)$ satisfies~(\ref{d.ann_1_of_ctb}),
and hence the desired conclusion~(\ref{d.f_factors}).
((\ref{d.f_factors}) only refers to factorization as a map of
abelian groups.
However, the maps $A\to A\,/\,\U_i$ are algebra homomorphisms,
hence we in fact get a factorization as an algebra homomorphism.)

Note that by~(\ref{d.mut_ann}), and the
surjectivity assumption of~(\ref{d.CC}), for any $J\subseteq I$
the ideals $f(\prod_{i\in J} A_i)+Z(B)$
and $f(\prod_{i\in I-J} A_i)+Z(B)$ of $B$ are almost direct factors.

Suppose, now, in contradiction to~(\ref{d.ann_1_of_ctb}), that
we had a partition of $I$ into subsets $J_m$ $(m\in\omega)$
such that none of the ideals $\prod_{J_m} A_i$
belonged to the kernel of $A\to B\to B/Z(B).$
I claim that the chain of almost direct factors
\begin{equation}\begin{minipage}[c]{35pc}\label{d.chain}
$Z(B)\ \subseteq
\ f(\prod_{i\in J_0} A_i)+Z(B)\ \subseteq\ \dots\ \subseteq
\ f(\prod_{i\in J_0\cup\dots\cup J_{n-1}} A_i)+Z(B)\ \subseteq\ \dots$
\end{minipage}\end{equation}
would be strictly increasing.
Indeed, the step where $J_n$ first comes in
cannot equal the preceding step,
because $f(\prod_{i\in J_n} A_i)$ annihilates
the latter, but not the former, by~(\ref{d.not_self}).
This would contradict the chain condition
assumed in~(\ref{d.CC}).
Thus~(\ref{d.ann_1_of_ctb}) holds, as claimed.

\subsection{Nonabelian groups}\label{SS.nonab}
If $K$ is a group, we can similarly call normal subgroups $K_0$
and $K_1$ ``almost direct factors'' of $K$ if each is the centralizer
of the other and their product is all of $K,$ and so define
chain condition on almost direct factors for groups.
The same reasoning as above, with
mutually centralizing normal subgroups in place of mutually
annihilating ideals, and products of normal
subgroups in place of sums of ideals, yields the analogous result.
Namely, letting $Z(K)$ now denote the center of~$K,$
the reader can verify that we get

\begin{proposition}\label{P.CC_gp}
If $f: H=\prod_I H_i\to K$ is a surjective homomorphism of
groups, and $K$ has chain condition on almost direct factors,
then there exist ultrafilters $\U_0,\dots,\U_{n-1}$ on $I$ such that
the composite map $H\to K\to K/Z(K)$ factors through the
natural map $H\to H/\,\U_0\times\dots\times H/\,\U_{n-1}.$\qed
\end{proposition}

\section{Further thoughts on the above results}\label{S.thoughts}

The statement and proof of Proposition~\ref{P.CC_gp} above are
exactly modeled on those of~(\ref{d.CC}),
but one proof uses properties specific to nonabelian groups,
the other, properties specific to $\!k\!$-algebras.
Can we set up a general context which embraces these two cases,
and leads to more examples?

Say we are working in a general variety $\V$ of algebras,
in the sense of universal algebra.
We have the minor complication that if $\V$ does not involve
a group structure, we lose the simplification of
interpreting the filter determined by a homomorphism
on a direct product algebra
via its kernel, as in~(\ref{d.{supp}}).
But I don't think this should make a big difference;
we still have~(\ref{d.A->C});
we must simply expect certain statements to involve
twice as many variables as when we have a
group structure, since we must deal with
the condition that two elements fall together under a map,
rather than the condition that one element be in the kernel.

A less trivial problem is what should replace
annihilators in algebras and centralizers in groups.
I think something like the following might work.

Let us understand a \emph{formal relation} in variables
$x_0,\dots,x_{n-1},$ written
\begin{equation}\begin{minipage}[c]{35pc}\label{d.rel}
$R(x_0,\dots,x_{n-1}),$
\end{minipage}\end{equation}
to mean a symbolic equation
\begin{equation}\begin{minipage}[c]{35pc}\label{d.rel12}
$R_0(x_0,\dots,x_{n-1})\ =\ R_1(x_0,\dots,x_{n-1}),$
\end{minipage}\end{equation}
where $R_0$ and $R_1$ are terms in $n$ variables
and the operations of $\V.$

Let us now consider formal relations
$R(x,x';y,y';z_0,\dots,z_{n-1})$ in $n+4$ variables,
which we will mostly abbreviate to $R(x,x';y,y'),$
suppressing the final $n$ variables, with the property that
\begin{equation}\begin{minipage}[c]{35pc}\label{d.R}
$R(x,x,y,y')$ and $R(x,x',y,y)$ are identities of $\V.$
\end{minipage}\end{equation}
Thus, (\ref{d.R}) says that our relation becomes an identity
of $\V$ if \emph{either} the two variables $x$ and $x',$ \emph{or}
the two variables $y$ and $y',$ are set equal.

(For example, if $\V$ is the variety of
$\!k\!$-algebras, then two examples of formal relations
satisfying~(\ref{d.R}) are $(x-x')(y-y')=0$ and $((x-x')z)(y-y')=0.$
In the variety of groups, examples are $[xx'^{-1},\,yy'^{-1}]=e$
and $[x^2x'^{-2},\,(yy'^{-1})^3]=e.$
For an example in a variety not involving a group
structure, we may take for $\V$ the variety of lattices,
and for $R$ the formal relation
\begin{equation}\begin{minipage}[c]{35pc}\label{d.lat_eg}
$(x\vee y)\wedge (x'\vee y')\ =\ (x\vee y')\wedge (x'\vee y).$
\end{minipage}\end{equation}
This last example generalizes to any variety with two
derived operations, denoted $\wedge,$ respectively, $\vee,$
each in two ``distinguished'' variables $x,\,y,$ and
possibly additional variables $z_0,\dots,z_{n-1},$
respectively $w_0,\dots,w_{m-1},$
such that $\wedge,$ but not necessarily $\vee,$ is commutative
in the distinguished variables.)

The following observation generalizes a property of products in
$\!k\!$-algebras, and commutators in groups,
that we used in the last two sections.
\begin{equation}\begin{minipage}[c]{35pc}\label{d.R_prod}
Suppose a formal relation $R(x,x',y,y')$ in the operations of a variety
$\V$ satisfies~(\ref{d.R}).
Then on a direct product $A=A_0\times A_1$ of algebras
in $\V,$ the relation $R(a,a',b,b')$ holds whenever
$a,\,a'$ agree in their $\!A_0\!$-components and
$b,\,b'$ agree in their $\!A_1\!$-components.
\end{minipage}\end{equation}
(Cf.\ the property of $\!k\!$-algebras that
if one element of $A_0\times A_1$ has zero first component,
and another has zero second component, then their product is zero.)
The idea is that such relations should help ``detect''
direct product decompositions.

Now let $\mathcal{R}_\V$ be the set of all formal relations
$R$ satisfying~(\ref{d.R}) in $\V.$
Then for any binary relation $C$ on the underlying set of an
algebra $A\in\V,$ we can define a binary relation $C^\perp,$ by
\begin{equation}\begin{minipage}[c]{35pc}\label{d.C^perp}
$C^\perp\ =\ \{(a,a')\in A\times A \mid$ for all
$R(x,x';y,y';z_0,\dots,z_{n-1})\in\mathcal{R}_\V,$ all pairs of
elements $(b,b')\in C,$ and all choices of $c_0,\dots,c_{n-1}\in A,$
the relation $R(a,a';b,b';c_0,\dots,c_{n-1})$ holds in $A\,\}.$
\end{minipage}\end{equation}

The hope is that this construction, applied to \emph{congruences} $C,$
will play a role analogous to \emph{annihilators of ideals} of
a $\!k\!$-algebra, and \emph{centralizers of subgroups}
of a group, and so allow us to prove a general analog of~(\ref{d.CC})
and Proposition~\ref{P.CC_gp}.
But exactly how this should be done is not clear.
For instance, though one can show that the relation $C^\perp$ defined
in~(\ref{d.C^perp}) will be a subalgebra of $A\times A,$
and as a binary relation it is easily seen to be reflexive and
symmetric, I see no reason why it should be transitive,
and hence a congruence on $A,$ even if $C$ was a congruence.

One can, of course, consider two congruences
$C_0$ and $C_1$ to have a relation analogous to being ``almost direct
factors'' of an algebra or a group if they simultaneously
satisfy $C_0^\perp=C_1$ and $C_1^\perp=C_0,$
are mutually commuting, and have for join the improper congruence.
But to even state
the analog of~(\ref{d.CC}) and Proposition~\ref{P.CC_gp}, one
needs to be able to say that the analog of $Z(B),$ namely
the relation $(B\times B)^\perp$ (where
$B\times B$ is the improper congruence on $B)$ is a congruence.

Perhaps one needs to find additional conditions on the variety $\V$
that make such conclusions hold; and/or replace~(\ref{d.C^perp})
by a construction using, not all of $\mathcal{R}_\V,$ but some
subset $\mathcal{R}$ with appropriate properties.
(Note, however, that for an arbitrary subset
$\mathcal{R}\subseteq\mathcal{R}_\V,$
some of the things I've noted hold for $\mathcal{R}_\V$
may fail: $C^\perp$ need not be symmetric or a subalgebra.)
I leave these ideas for others to investigate.

Incidentally, not all situations to which we have applied the results
of \S\S\ref{S.cap_ultra}-\ref{S.A->C} are
based on ideas like those of ``annihilator'' and ``centralizer'',
whose possible generalization we have just examined.
The facts that $\Z$ and various other modules are slender
are true for (so far as I can see) very different sorts of reasons.

\section{Extending the Erd\H{o}s-Kaplansky Theorem to reduced products}\label{EK}

The Erd\H{o}s-Kaplansky Theorem \cite[Theorem~IX.2, p.247]{NJ}
says that if $D$ is a division ring and $I$ an infinite set,
then $\dim_D D^I=\card(D^I).$
Combining the method of proof of that result with Lemma~\ref{L.kappa}
above, we shall now prove a statement of which that theorem
is essentially the case $\F=\{I\}.$
(``Essentially'' because, to avoid complications,
we here assume $D$ infinite.
The case of finite $D$ will covered by Corollary~\ref{C.finite_D}.)

\begin{theorem}\label{T.E-K}
Let $D$ be an infinite division ring, $I$ a set, and $\F$ a filter
on $I$ which is not the intersection of finitely many
$\!\card(D)^+\!$-complete ultrafilters.
\textup{(}So in particular, $\F$ is not the principal
filter generated by a finite set.\textup{)}
Then
\begin{equation}\begin{minipage}[c]{35pc}\label{d.E-K}
$\dim_D\,D^I/\F\ =\ \card(D^I/\F).$
\end{minipage}\end{equation}
\textup{(}where $\dim_D$ can be taken to mean the
dimension either as a right or as a left vector space\textup{)}.
\end{theorem}

\begin{proof}
Without loss of generality, let vector spaces
(in particular, $D^I/\F)$ be left vector spaces.

For any vector space $V,$ since a basis for $V$
is a subset of $V,$ we have
\begin{equation}\begin{minipage}[c]{35pc}\label{d.dim_V}
$\dim\ V\ \leq\ \card(V).$
\end{minipage}\end{equation}
For $D$ infinite, as in the present situation, equality is
easily proved in~(\ref{d.dim_V}) if $\card(V)>\card(D).$
(E.g., by \cite[Lemma~X.1, p.245]{NJ}, and the fact
that a product of two infinite cardinals equals the
larger of the two.)
So to prove~(\ref{d.E-K}) it suffices to show that under our
assumptions on $\F,$ if $\card(D^I/\F)=\card(D),$ then
\begin{equation}\begin{minipage}[c]{35pc}\label{d.dim>_card_D}
$\dim_D(D^I/\F)\ \geq\ \card(D).$
\end{minipage}\end{equation}

To do this, let us slightly strengthen \cite[Lemma~IX.2, p.246]{NJ}.
Namely, we shall show that there exists a
$\!\card(D)\times \card(D)\!$-tuple
$((x_{\alpha\beta}))_{\alpha,\beta\in\card(D)}$
of elements of $D$ such that
\begin{equation}\begin{minipage}[c]{35pc}\label{d.strong}
For every natural number $n,$ and
every pair of $\!n\!$-tuples $(\alpha_0,\dots,\alpha_{n-1})$
and $(\beta_0,\dots,\beta_{n-1})$ of elements of $\card(D)$ such
that $\alpha_0<\dots<\alpha_{n-1}$
and $\beta_0<\dots<\beta_{n-1},$ the $n$
elements $(x_{\alpha_i\beta_j})_{i=0,\dots,n-1}\in D^n$
$(j=0,\dots,n-1)$ are linearly independent; equivalently,
the $n\times n$ matrix $((x_{\alpha_i\beta_j}))$ is nonsingular.
\end{minipage}\end{equation}
(Such a family is called \emph{strongly} linearly independent
in \cite[Lemma~X.2, p.246]{NJ}, though there, the index we
call $\beta$ is restricted to a countable range;
i.e., only countable strongly linearly independent families
are considered.)

Mimicking the proof in \cite{NJ},
we choose the elements $x_{\alpha\beta}\in D$
by a recursion over the index-set $\card(D)\times\card(D),$
lexicographically ordered.
Given $\alpha,\,\beta,$ assume recursively that
all $x_{\alpha'\beta'}$ with $(\alpha',\beta')<(\alpha,\beta)$
have been chosen so as to satisfy all cases
of~(\ref{d.strong}) involving only elements with
subscripts $<(\alpha,\beta).$
In particular, for every natural number $n$ and
every pair of increasing $\!n\!$-tuples $(\alpha_0,\dots,\alpha_{n-1})$
and $(\beta_0,\dots,\beta_{n-1})$ with $\alpha_{n-1}=\alpha,$
$\beta_{n-1}=\beta,$ the values $x_{\alpha_i\beta_j}$
other than $x_{\alpha_{n-1}\beta_{n-1}}$ have been chosen;
so we have an $n\times n$ matrix with one entry missing; and
by our recursive assumption, its upper left $n{-}1\times n{-}1$
minor is nonsingular.
In this situation, one sees by linear algebra
that one and only one value of the missing element will
make the matrix singular.
(Indeed, a unique linear combination of the first $n-1$
rows will have first $n-1$ coordinates agreeing with those specified
in the $\!n\!$-th row; and the last coordinate of the resulting row
will be the value of $x_{\alpha\beta}$ in question.)

Now since $\alpha,\,\beta<\card(D),$ there are fewer than
$\card(D)\,\card(D)=\card(D)$ choices for the integer $n$ and the values
$\alpha_0,\dots,\alpha_{n-2}$ and $\beta_0,\dots,\beta_{n-2}$
in the preceding paragraph.
Since each such choice leads to only one value of
$x_{\alpha\beta}$ making the corresponding matrix singular, we may
choose $x_{\alpha\beta}$ so as to make all these matrices nonsingular.
Proceeding recursively, we get values of $x_{\alpha\beta}$ for all
$\alpha,\beta\in\card(D)$ which together satisfy~(\ref{d.strong}).

Now by assumption, our filter $\F$ is not a finite
intersection of $\!\card(D)^+\!$-complete
ultrafilters; so Lemma~\ref{L.kappa} tells us that there
exists a partition of $I$ into $<\card(D)^+,$ i.e.,
$\leq\card(D)$ subsets no finite union of which belongs to $\F.$
If that partition involves fewer than $\card(D)$
sets, let us throw in empty sets to reach that value.
Thus, we can write our partition $(J_\alpha)_{\alpha\in\card(D)}.$
Let us now define elements $y_\beta\in D^I$ $(\beta\in\card(D))$
by the conditions that on each $J_\alpha,$ the element $y_\beta$ has
constant value $x_{\alpha\beta}.$

Then~(\ref{d.strong}) tells us that for any positive integer $n,$ a
nontrivial linear combination of $n$ of these elements
cannot be zero on $n$ of the sets $J_\alpha;$
i.e., its zero-set must be a union of $<n$ of those sets.
So as no union of finitely many $J_\alpha$ belongs
to $\F,$ no nontrivial linear combination of
the $y_\beta$ has zero image in $D^I/\F.$
Thus, we have a $\!\card(D)\!$-tuple of linearly independent
elements of $D^I/\F,$ proving~(\ref{d.dim>_card_D}), as required.
\end{proof}

In contrast, if $\F$ \emph{is} an intersection of $n\geq 0$
$\!\card(D)^+\!$-complete ultrafilters,
then $D^I/\F\cong D^n,$ which has dimension less than its cardinality.
\vspace{.5em}

{\em Remarks for the reader familiar with \cite{prod_Lie2}}:
Let me note how,
with the help of the above theorem, one can strengthen some of the
results of \cite{prod_Lie2} from the case of vector spaces
over a field $k$ to that of vector spaces over a division ring $D.$
We begin with \cite[Lemma~7]{prod_Lie2}.
(I will not to repeat here the statement of that technical
result, but merely note how to extend its proof.
I will, however, recall something of the content of the results
that follow from that lemma.)
One finds that all steps of the proof of that lemma \emph{except}
the paragraph following~\cite[(38)]{prod_Lie2}
go over unchanged to the division ring case.
The result of that paragraph says (after putting $D$ for $k)$ that
a certain ultrafilter $\U$ on a certain set $J$
is $\!\card(D)^+\!$-complete~\cite[(39)]{prod_Lie2}.
This is vacuous if $D$ is finite; to prove it when $D$ is infinite, note
that~\cite[(38)]{prod_Lie2}
says that $D^J/\U$ can be mapped injectively into the
finite-dimensional $\!D\!$-vector-space $g(D^{J_0\cup J})/g(D^{J_0});$
so $D^J/\U$ is finite-dimensional.
Thus, by Theorem~\ref{T.E-K} above, $\U$
is a finite intersection of $\!\card(D)^+\!$-complete ultrafilters,
which, given that it is an ultrafilter, simply says it is
$\!\card(D)^+\!$-complete, as desired.

From this we get the corresponding generalization of
the corollary to that lemma, which in particular tells
us that for $D$ infinite and $I$ a set having cardinality
less than every measurable cardinal $>\card(D),$
every subspace of finite codimension in $D^I$ contains
an element of cofinite support in $I.$

This result, in turn, allows us to generalize
\cite[Theorem~9(iii)]{prod_Lie2}, i.e., roughly, case~(iii)
of~(\ref{d.prod_Lie2}) above, to the situation
where $\!k\!$-algebras $A$ are replaced by $\!(D,D)\!$-bimodules
$A$ given with balanced $\!D\!$-bilinear maps $A\times A\to A.$
In the generalized statement, one should
assume both left and right finite-dimensionality of $B.$
One adapts the supporting result \cite[Lemma~2]{prod_Lie2} by
associating to each $a=(a_i)\in \prod_I A_i$ both the left-vector-space
map $g_a:D^I\to B$ defined by $g_a((u_i))=f((u_i a_i))$
and the right-vector-space
map $g'_a:D^I\to B$ defined by $g'_a((v_i))=f((a_i v_i)).$
In display~(5) of the proof of that lemma, the key step
becomes $f((a_i\,x_i))=f((a_i\,v_i\,v_i^{-1}x_i)),$
and in the dual calculation,
$f((x_i\,a_i))=f((x_i\,u_i^{-1}\,u_i\,a_i)).$

(In this version of \cite[Theorem~9(iii)]{prod_Lie2},
I wonder whether the two-sided finite-dimensionality
assumption can somehow be weakened to one-sided finite-dimensionality,
or even ascending chain condition on
the $\!(D,D)\!$-bimodule structure of~$B.)$ \vspace{.5em}

Returning to Theorem~\ref{T.E-K}, of which, we saw, the only
nontrivial case was when $\card(D^I/\F)=\card(D),$ it is
natural to ask how common this equality is -- in other words, how
common it is for an
infinite set $X$ to satisfy $\card(X^I/\F)=\card(X)$ when
the filter $\F$ is not a finite intersection of
$\!\card(X)^+\!$-complete ultrafilters.

For some values of $\card(X),$ it is indeed common.
For instance, if $\kappa=\lambda^\mu,$ where $\lambda$
and $\mu$ are infinite cardinals, then $\kappa^\mu=\kappa,$
hence for any $X$ of cardinality $\kappa,$ and nonempty $I$
of cardinality $\leq\mu,$ we have $\card(X^I)=\card(X);$
hence for any proper filter $\F$ on $I,$
$\card(X)=\card(X^I)\geq\card(X^I/\F)\geq\card(X)$
(the last inequality
because $X$ embeds diagonally in $X^I/\F),$ giving the desired equality.

On the other hand, if $X$ is countably infinite, we always
get $\card(X^I/\F)\geq 2^{\aleph_0}.$
For, following the idea of the proof of
\cite[Lemma~6]{prod_Lie2}, consider the continuum many functions
$f_r:\omega\to\omega$ given by $f_r(n)=\lfloor r\,n\rfloor$
for positive real numbers $r.$
We see that any two of these agree only at finitely many $n.$
Now given a filter $\F$ on $I$ that is not countably complete, we can,
by the same sort of application of Lemma~\ref{L.kappa} as
in the next-to-last paragraph of
the proof of Theorem~\ref{T.E-K} above, construct from the $f_r$
continuum many functions $y_r:I\to\omega$ that have distinct
images in $\omega^I/\F;$ and using a bijection between $X$
and $\omega,$ we get the asserted conclusion.

I do not know whether one of the above two examples is more ``typical''
than the other.

It is also natural to ask,
when will a reduced power of a finite set be finite?
The answer (and a bit more) is given in

\begin{lemma}\label{L.red_finite}
Let $\F$ be a filter on a set $I,$ and $X$ a finite set
with more than one element.
Then the equivalent conditions of Lemma~\ref{L.finiteU}
\textup{(}in particular, condition\textup{~(\ref{d.cap_fin})},
that $\F$ is the intersection of finitely many ultrafilters on $I)$
are also equivalent to each of
\begin{equation}\begin{minipage}[c]{35pc}\label{d.red_finite}
The reduced power $X^I/\F$ is finite,
\end{minipage}\end{equation}
\begin{equation}\begin{minipage}[c]{35pc}\label{d.red<c}
The reduced power $X^I/\F$ has cardinality $<2^{\aleph_0}.$
\end{minipage}\end{equation}
\end{lemma}

\begin{proof}
If~(\ref{d.cap_fin}) holds, with $\F=\U_0\cap\dots\cap\U_{n-1},$
then $X^I/\F$ embeds in
$X^I/\U_0\times\dots\times X^I/\U_{n-1},$ and it is well
known that an ultrapower of a finite set $X$ is isomorphic to $X;$
so $X^I/\F$ is finite, giving~(\ref{d.red_finite})
and hence~(\ref{d.red<c}).

Conversely, if the equivalent conditions of
Lemma~\ref{L.finiteU} fail, then the failure of~(\ref{d.1_of_ctb})
says that we have a partition of $I$ into
subsets $J_m$ $(m\in\omega)$ none of whose complements lies in $\F.$
In this case, let us take two elements $x\neq y\in X,$ and
consider the $2^{\aleph_0}$ elements of $\{x,y\}^I$
which are constant on each subset $J_m.$
If $f$ and $g$ are distinct members of this set, then they
disagree throughout at least one $J_m;$ so as the complement of $J_m$
does not lie in $\F,$ $f$ and $g$ yield distinct elements of $X/\F.$
This contradicts~(\ref{d.red<c}), and hence~(\ref{d.red_finite}).
\end{proof}

\begin{corollary}\label{C.finite_D}
In Theorem~\ref{T.E-K}, the condition that $D$ be infinite
can be dropped.
\end{corollary}

\begin{proof}
Assume $D$ finite.
Thus, the hypothesis that $\F$ is not the intersection
of finitely many $\!\card(X)^+\!$-complete ultrafilters
simply says it is not the intersection of finitely many ultrafilters,
which by Lemma~\ref{L.red_finite} tells us that the vector space
$D^I/\F$ is infinite.
But for an infinite vector space $V$
over a finite field, one indeed has equality in~(\ref{d.dim_V}).
\end{proof}

\end{document}